\documentclass[12pt]{amsart}
\usepackage{amsmath}
\usepackage{amsfonts}
\usepackage{amssymb}

\usepackage{tikz}
\usetikzlibrary{arrows}
\usetikzlibrary{matrix}
\makeatletter
\def\@seccntformat#1{%
  \protect\textup{%
    \protect\@secnumfont
    \expandafter\protect\csname format#1\endcsname 
    \csname the#1\endcsname
    \protect\@secnumpunct
  }%
}

\begin{document}
\title[zero coefficients of the powers of the determinant]
{
A condition for the existence of \\
zero coefficients in \\
the powers of the determinant polynomial}
\author[M. Itoh]{Minoru Itoh}
\address{Department of Mathematics and Computer Science, 
         Faculty of Science,
         Kagoshima University, Kagoshima 890-0065, Japan}
\email{itoh@sci.kagoshima-u.ac.jp }
\author[J. Shimoyoshi]{Jimpei Shimoyoshi}
\address{OKI Software Co., Ltd., Saitama 335-8510, Japan }
\email{k6434019@kadai.jp}
\date{}
\begin{abstract}
We discuss the existence of zero coefficients 
in the powers of the determinant polynomial of order $n$.
D. G. Glynn proved that the coefficients of the $m$th power of the determinant polynomial are all nonzero,
if $m = p-1$ with a prime $p$.
We show that the converse also holds, if $n \geq 3$.
The proof is quite elementary.
\end{abstract}
\thanks{
This research was partially supported by JSPS Grant-in-Aid for Scientific Research (C) 16K05067.}
\maketitle
\theoremstyle{plain}
   \newtheorem{theorem}{Theorem}[section]
   \newtheorem{proposition}[theorem]{Proposition}
   \newtheorem{lemma}[theorem]{Lemma}
   \newtheorem{corollary}[theorem]{Corollary}
\theoremstyle{definition}
   \newtheorem{definition}[theorem]{Definition}
   \newtheorem{conjecture}[theorem]{Conjecture}
\theoremstyle{remark}
   \newtheorem*{remark}{Remark}
   \newtheorem*{remarks}{Remarks}
\numberwithin{equation}{section}
%

%
\section{Introduction}
%
%
We consider the expansion of the powers of the determinant polynomial,
and discuss the existence of zero coefficients.
D. G. Glynn proved that the coefficients of the $m$th power of the determinant polynomial of order $n$
are all nonzero,
if $m = p-1$ with a prime $p$.
This result is remarkable because this leads a proof of 
the Alon--Tarsi conjecture in dimension $p-1$.
In this article,
we show that the converse of Glynn's result also holds, if $n \geq 3$.
The proof is quite elementary.

Let us explain the assertion precisely.
Let $X = (x_{ij})_{1 \leq i,j \leq n}$ be 
an $n$ by $n$ matrix whose entries are indeterminates.
We define the coefficients $C_L$ by the following expansion of $(\det X)^m$:
\begin{equation}
   (\det X)^m = \sum_{L \in \Psi(m)} C_L x^L.
\label{eq:definition of C_L}
\end{equation}
Here $\Psi(m)$ is the set of all $n$ by $n$ matrices
of nonnegative integers with each row and column summing to $m$:
\[
   \Psi(m) = \left\{ (l_{ij})_{1 \leq i,j \leq n}  
   \,\left|\,
   \begin{aligned}
   &l_{ij} \in \mathbb{Z}_{\geq 0}, \\
   &\textstyle\text{$\sum_{i=1}^n l_{ij} = m$ for any $j = 1,2,\ldots,n$}, \\
   &\textstyle\text{$\sum_{j=1}^n l_{ij} = m$ for any $i = 1,2,\ldots,n$} 
   \end{aligned}
   \right.
   \right\}.
\]
Moreover, 
we put $x^L = \prod_{1 \leq i,j \leq n} x_{ij}^{l_{ij}}$ for $L = (l_{ij})_{1 \leq i,j \leq n}$.

D. G. Glynn proved the following theorem for these coefficients $C_L$:

\begin{theorem}\label{thm:when m = p-1}\slshape
   If $p$ be prime, 
   we have $C_L \ne 0$ for all $L \in \Psi(p-1)$.
\end{theorem}

\begin{remark}
   Actually, Glynn proved a stronger theorem,
   namely, that we have $L! C_L \equiv (-1)^n \pmod p$ for all $L \in \Psi(p-1)$. 
   Here we put $L! = \prod_{i,j=1}^n l_{ij}!$ for $L = (l_{ij})_{1 \leq i,j \leq n}$.
\end{remark}

In the present article, we prove that the inverse of Theorem~\ref{thm:when m = p-1} also holds, when $n \geq 3$:

\begin{theorem}\label{thm:when m ne p-1}\slshape
   Assume that $n \geq 3$.
   Let $m$ be a natural number which cannot be expressed as $m = p-1$ with a prime $p$.
   Then, there exists $L \in \Psi(m)$ satisfying $C_L = 0$.
\end{theorem}

Thus, we see that the following two conditions on $m$ are equivalent, when $n \geq 3$:
\begin{itemize}
\item We have $C_L \ne 0$ for all $L \in \Psi(m)$.
\item There is a prime $p$ satisfying $m=p-1$.
\end{itemize}

%
\section{The coefficients are all nonzero, when $m = p-1$}
%
%
Theorem~\ref{thm:when m = p-1} was shown by D. G. Glynn in \cite{G2}.
The key of the proof is the hyperdeterminant
introduced by Glynn himself in \cite{G1}
(this hyperdeterminant occurs only for fields of prime characteristic $p$).
Nowadays, a proof of Theorem~\ref{thm:when m = p-1} 
without hyperdeterminant is also known \cite{K1}, \cite{K2}.

Theorem~\ref{thm:when m = p-1} is remarkable,
because this leads to a special case of the Alon--Tarsi conjecture on Latin squares.
Let $\operatorname{els}(n)$ and $\operatorname{ols}(n)$
denote the numbers of even and odd Latin squares of size $n$, respectively.
We can easily show $\operatorname{els}(n) = \operatorname{ols}(n)$ 
when $n$ is an odd number greater than $1$.
In contrast, on the case that $n$ is even,
the following conjecture was proposed by N. Alon and M. Tarsi \cite{AT}:

\begin{conjecture}
   When $n$ is even, we have
   $\operatorname{els}(n) \ne \operatorname{ols}(n)$.
\end{conjecture}

Glynn proved that this conjecture is true,
when $n = p-1$ with a prime $p$:

\begin{theorem}\label{thm:AT conjecture for p-1}\slshape
   For any prime $p$, we have
   $\operatorname{els}(p-1) \ne \operatorname{ols}(p-1)$.
\end{theorem}

This is deduced from Theorem~\ref{thm:when m = p-1}
by looking at the coefficient corresponding to the all-ones matrix 
$J_n = (1)_{1 \leq i,j \leq n} \in \Psi(n)$.
Indeed we have the relation 
\[
   C_{J_n} = 
   (-)^{n(n-1)/2} 
   (\operatorname{els}(n) - \operatorname{ols}(n))
\]
as a corollary of the relation 
between the ordinary parity and the symbol parity of Latin squares given in \cite{J}.

\begin{remark}
   The Alon--Tarsi conjecture is also proved when $n=p+1$ with an odd prime $p$
   \cite{D}.
   This conjecture is related to various problems including Rota's basis conjecture.
   See \cite{FM} for the results related to this conjecture.
\end{remark}

%
\section{There exists a zero coefficient, when $m \ne p-1$}
%
%
Let us prove Theorem~\ref{thm:when m ne p-1}.
This theorem was first found in Master's thesis of the second author \cite{S}.

Let $m$ be a natural number which cannot be expressed as $m = p-1$ with a prime $p$.
We can specifically find $L \in \Psi(m)$ satisfying $C_L = 0$ as follows. 
When $n=3$, 
we consider the following $3$ by $3$ matrix:
\[
   L_3(a,b) = 
   \begin{pmatrix}
   ab+b-1 & a & 1 \\
   a & ab & b \\
   1 & b & ab+a-1
   \end{pmatrix}.
\]
Here, $a$ and $b$ are natural numbers satisfying $(a+1)(b+1) = m+1$
(there exist such $a$ and $b$, because $m+1$ is a composite number).
For general $n \geq 3$, we consider the following $n$ by $n$ matrix:
\[
   L_n(a,b) = 
   \begin{pmatrix}
   L_3(a,b) &  & & \\
    & m & & & \\
    & & \ddots & & \\
    & & & & m
   \end{pmatrix}.
\]
Then the coefficient corresponding to this matrix is zero:

\begin{proposition}\slshape
   When $n \geq 3$,
   we have $C_{L_n(a,b)} = 0$.
\end{proposition}

This proposition follows from the following two lemmas.
Firstly, $C_{L_n(a,b)}$ is expressed as the difference of two multinomial coefficients:

\begin{lemma}\label{lem:difference of two multinomial coefficients}\slshape
   We have 
   \[
      C_{L_n(a,b)} = 
      (-)^{a+b+1} {m \choose ab-1,0,0,a,b,1}
      +(-)^{a+b}{m \choose ab,1,1,a-1,b-1,0},
   \]
   where
\[
   {m \choose m_1,m_2,m_3,m_4,m_5,m_6} = \frac{m!}{m_1! m_2! m_3! m_4! m_5! m_6!}.
\]
\end{lemma}

Secondly, these two multinomial coefficients are equal to each other:

\begin{lemma}\label{lem:two multinomial coefficients are equal}\slshape
   We have
   \[
      {m \choose ab-1,0,0,a,b,1}
      = {m \choose ab,1,1,a-1,b-1,0}.
   \]
\end{lemma}

Lemma~\ref{lem:two multinomial coefficients are equal} follows by a direct calculation,
and Lemma~\ref{lem:difference of two multinomial coefficients} is proved as follows:

\begin{proof}[Proof of Lemma~\ref{lem:difference of two multinomial coefficients}]
First we consider the case of $n=3$.
We put
\begin{align*}
   \alpha_1 &= x_{11} x_{22} x_{33}, &
   \alpha_2 &= x_{12} x_{23} x_{31}, &
   \alpha_3 &= x_{13} x_{21} x_{32}, \\
   \beta_1 &= x_{12} x_{21} x_{33}, &
   \beta_2 &= x_{11} x_{23} x_{32}, &
   \beta_3 &= x_{13} x_{22} x_{31},
\end{align*}
such that
\[
   \det X = \alpha_1 + \alpha_2 + \alpha_3 - \beta_1 - \beta_2 - \beta_3,
\]
and $(\det X)^m$ is expanded as follows:
\begin{equation}
   (\det X)^m = \sum_{k_1 + k_2 + k_3 + l_1 + l_2 + l_3 = m}
   (-)^{l_1+l_2+l_3} {m \choose k_1,k_2,k_3,l_1,l_2,l_3}
   \alpha_1^{k_1} \alpha_2^{k_2} \alpha_3^{k_3} \beta_1^{l_1} \beta_2^{l_2} \beta_3^{l_3}.
\label{eq:multinomial expansion}
\end{equation}
Let us determine all $6$-tuples $(k_1,k_2,k_3,l_1,l_2,l_3)$ of nonnegative integers 
satisfying the following relation:
\begin{equation}
   \alpha_1^{k_1} \alpha_2^{k_2} \alpha_3^{k_3} \beta_1^{l_1} \beta_2^{l_2} \beta_3^{l_3}
   = x^{L_3(a,b)}.
\label{eq:x^L_3}
\end{equation}
Since the left hand side can be expressed as
\begin{align*}
   \alpha_1^{k_1} \alpha_2^{k_2} \alpha_3^{k_3} \beta_1^{l_1} \beta_2^{l_2} \beta_3^{l_3}
   & = 
   x_{11}^{k_1 + l_2}  
   x_{12}^{k_2 + l_1}  
   x_{13}^{k_3 + l_3} \\   
   & \phantom{{}={}} 
   x_{21}^{k_3 + l_1}  
   x_{22}^{k_1 + l_3}  
   x_{23}^{k_2 + l_2} \\
   & \phantom{{}={}} 
   x_{31}^{k_2 + l_3}  
   x_{32}^{k_3 + l_2} 
   x_{33}^{k_1 + l_1},
\end{align*}
this relation is equivalent with the following system of nine linear equations:
\[
   \begin{pmatrix}
   k_1 + l_2 & k_2 + l_1 & k_3 + l_3 \\   
   k_3 + l_1 & k_1 + l_3 & k_2 + l_2 \\
   k_2 + l_3 & k_3 + l_2 & k_1 + l_1    
   \end{pmatrix} 
   = L_3(a,b)
   = 
   \begin{pmatrix}
   ab+b-1 & a & 1 \\
   a & ab & b \\
   1 & b & ab+a-1
   \end{pmatrix}.
\]
Solving this, we see that
$(k_1,k_2,k_3,l_1,l_2,l_3) \in \mathbb{Z}_{\geq 0}^6$ satisfying (\ref{eq:x^L_3}) are
\[
   (ab-1,0,0,a,b,1), \qquad
   (ab,1,1,a-1,b-1,0).
\]
Therefore, comparing (\ref{eq:definition of C_L}) and (\ref{eq:multinomial expansion}),
we have
\[
   C_{L_3(a,b)} = 
   (-)^{a+b+1}   {m \choose ab-1,0,0,a,b,1}
   +(-)^{(a-1) + (b-1) + 0} {m \choose ab,1,1,a-1,b-1,0},
\]
namely the assertion in the case of $n=3$.

The case of $n > 3$ is also almost the same. 
Indeed,
to calculate $C_{L_n(a,b)}$,
we need to look at the following relation instead of (\ref{eq:x^L_3}):
\[
   \prod_{1\leq i\leq m}x_{1\sigma_i(1)}x_{2\sigma_i(2)}\cdots x_{n\sigma_i(n)} = x^{L_n(a,b)}.
\]
Since $\sigma_1, \ldots, \sigma_m$ satisfying this relation belong to
\[
   \{ \sigma \in S_n \,|\, \text{$\sigma(k) = k$ for any $k = 4,5,\ldots,n$}\} \simeq S_3,
\]
the proof is reduced to the case of $n=3$.
\end{proof}

%
%
%

\end{document}